\documentclass[12pt,reqno]{amsart}
\usepackage{setspace}
\usepackage{amsmath, amssymb, amscd, amsthm, amsfonts}
\usepackage{graphicx}
\usepackage{enumitem}
\usepackage{mathdots}
\usepackage{mathpazo}

\usepackage{indentfirst}
\usepackage{amssymb,amsmath,amsfonts,amsthm}
\usepackage[all]{xy}
\usepackage{enumerate}
\usepackage{mathrsfs}
\usepackage{graphicx}
\numberwithin{equation}{section}
\usepackage[all]{xy}
\usepackage{amssymb,amscd,amsthm,amsmath,graphicx,color}
\usepackage{mathrsfs}
\usepackage[english]{babel}
\usepackage[utf8x]{inputenc}

\newtheorem{theorem}{Theorem}[section]
\newtheorem*{proposition*}{Proposition}
\newtheorem{proposition}[theorem]{Proposition}
\newtheorem{corollary}[theorem]{Corollary}
\newtheorem{lemma}[theorem]{Lemma}
\newtheorem{example}[theorem]{Example}
\newtheorem{remark}[theorem]{Remark}
\newtheorem{definition}[theorem]{Definition}
\newtheorem{question}[theorem]{Question}

\newtheorem{notation}[theorem]{Notation}

\DeclareMathOperator{\Hom}{Hom}
\DeclareMathOperator{\Ext}{Ext}
\DeclareMathOperator{\Tor}{Tor}

\DeclareMathOperator{\coker}{coker}

\DeclareMathOperator{\Supp}{Supp}
\DeclareMathOperator{\depth}{depth}

\DeclareMathOperator{\pd}{pd}
\DeclareMathOperator{\cd}{cd}

\DeclareMathOperator{\Ass}{Ass}

\DeclareMathOperator{\ann}{ann}
\DeclareMathOperator{\im}{im}
\DeclareMathOperator{\Gpd}{Gpd}

\title[Generalized Hartshorne's problem]{Generalized Hartshorne's problem on finiteness properties of local cohomology modules}

\author[A. Dosea]{Andr\'e Dosea}
\address{Departamento de Matemática, Universidade Federal da Paraíba - 58051-900, João Pessoa, PB, Brazil}
\email{andredosea@hotmail.com}

\author[\,R. Holanda]{\,Rafael Holanda}
\address{SISSA (Scuola Internazionale Superiore di Studi Avanzati), Via Bonomea 265, 34136 Trieste, Italy; and
Departamento de Matemática, Universidade Federal da Paraíba - 58051-900, João Pessoa, PB, Brazil}
\email{rferreir@sissa.it, rf.holanda@gmail.com}

\author[\,C. B. Miranda-Neto]{\,Cleto B. Miranda-Neto}
\address{Departamento de Matemática, Universidade Federal da Paraíba - 58051-900, João Pessoa, PB, Brazil}
\email{cleto@mat.ufpb.br}

\date{\today}


\begin{document}

\maketitle

\begin{abstract} Our main goal in this paper is to answer new positive cases of the natural generalized version of Hartshorne's celebrated question on cofiniteness of local cohomology modules, and consequently of Huneke's conjecture on the finiteness of their sets of associated primes. Our approach, by means of which we extend several results from the literature, is essentially based on spectral sequence techniques and connections to numerical invariants such as the cohomological dimension and the Gorenstein projective dimension. We also provide, over a polynomial ring, a rather pathological example of a non weakly Laskerian module (i.e., it admits a quotient with infinitely many associated primes) whose first local cohomology module is non-zero and has finitely many associated primes.

  
\end{abstract}


\section{Introduction}


In 1968, Grothendieck \cite{G} conjectured that, if $R$ is a (commutative, unital, Noetherian) ring, $I$ is an ideal of $R$, and $N$ is a finite (i.e., finitely generated) $R$-module, then the homomorphism module $\Hom_R (R/I, H^i_{I}(N))$ is  finite for all $i$. Here, $H^i_{I}(N)$ denotes the $i$-th local cohomology module of $N$ with respect to ${I}$. This conjecture is false and the first counterexample was given about two years later by Hartshorne \cite{H}, where in particular he defined an $R$-module $H$ to be $I$-\textit{cofinite} if $\Supp_RH \subset V(I)$ and the $R$-module $\Ext^i_R (R/I, H)$ is finite for all $i$. He then posed the following question which extends Grothendieck's conjecture.

\medskip

\noindent {\it Hartshorne's question}: If $N$ is finite, when is $H^i_{I}(N)$ $I$-cofinite for all $i\geq 0$?

\medskip


It is well-known and easy to see that $I$-cofinite modules have finitely many associated primes. This fact establishes a direct connection between Hartshorne's question and the charming conjecture regarding the finiteness of $\Ass_RH^i_{I}(N)$ raised by Huneke \cite{Hunekeconjecture} at Sundance Conference in 1990. 


\medskip

\noindent {\it Huneke's conjecture}: If $R$ is local and $N$ is finite, then for all $i \geq 0$ the module $H^i_{I}(N)$ has finitely many associated primes.

\medskip


There exist counterexamples to this conjecture (see Katzman \cite{Ka}, also Singh and Swanson \cite{SS}). On the other hand, for example,
Brodmann and Lashgari \cite{BL} showed that if $N$ is a finite $R$-module
such that $H^i_{I}(N)$ is finite for all $i<t$ then $\Ass_RH^t_{I}(N)$
is finite. It is also worth recalling that, if $R$ is not required to be local, these sets of associated primes may be infinite (see Singh \cite{Si, Si2}), whereas they are always finite if $R$ is a smooth ${\mathbb Z}$-algebra (see  Bhatt, Blickle, Lyubeznik, Singh, and Zhang \cite{BBLSZ}).

Such problems have attracted the attention of numerous researchers and have been answered
positively in several cases, typically in small dimensions; see  Bahmanpour and Naghipour \cite{BAH1}, Huneke and Koh \cite{Huneke}, Marley \cite{Marley}, and Melkersson \cite{Mel}. For excellent surveys on finiteness problems in local cohomology, we refer to Hochster 
\cite{Hochster} and Lyubeznik \cite{Lyu} as well as many of their references such as Huneke and Lyubeznik \cite{Huneke-L}, Huneke and Sharp \cite{Huneke-S}, and Lyubeznik \cite{Lyu0}.

Now let us recall that, making use of the weakly Laskerian property (which is satisfied, e.g., by all finite modules), Divaani-Aazar and Mafi \cite{DAM} introduced the class of the so-called $I$-weakly cofinite modules, which (strictly) extends the class of $I$-cofinite modules and inherits the property of having finitely many associated primes. Then, Hartshorne's question has been considered for $I$-weakly cofinite modules as well. In addition, the classical local cohomology module $H^i_{I}(N)$ was generalized by Herzog \cite{Herzog} to a suitable module $H^i_I(M,N)$, where $M$ is an $R$-module, in such a way that by letting $M=R$ we recover $H^i_{I}(N)$; for compelling reasons and several results that in our view justify the interest in {\it generalized} local cohomology (including a characterization of Cohen-Macaulay modules, the development of a general theory of deficiency and canonical modules, and criteria for prescribed bound on projective dimension), we refer to the recent paper \cite{FJMS}; in addition, as will be explained later in this Introduction, several of our results lead us to facts that are new even in the ordinary theory of local cohomology. So, there is the following interesting question.

\medskip

\noindent {\it Generalized Hartshorne's question}: If $M$ is finite and $N$ is weakly Laskerian, when is $H^i_{I}(M, N)$ $I$-weakly cofinite for all $i\geq 0$? Note this question can also be raised concerning the stronger property of $I$-cofiniteness, whereas now $M$ and $N$ are both finite.

\medskip

Likewise, Huneke's conjecture can be extended to the following problem.

\medskip

\noindent {\it Generalized Huneke's question}: If $M$ is finite and $N$ is weakly Laskerian, when is the set $\Ass_RH^i_{I}(M, N)$ finite for all $i\geq 2$? 


\medskip

Here it is worth observing that the set $\Ass_RH^1_{I} (M,N)$ is known to be finite; this follows from Mafi \cite[Theorem 3.3]{AM}. As to the module $H^0_{I} (M,N)$, it is easily seen to be $I$-weakly cofinite (this will be checked as part of some proof later in this paper) and hence $\Ass_RH^0_{I} (M,N)$ must be finite as well. Naturaly, a weaker related question to Huneke's is concerned with the finiteness of $\Ass_RH^i_{I}(M, N)$ for {\it some} $i\geq 2$ such that $H^i_{I}(M, N)\neq 0$.




The present work aims essentially at tackling the above generalized questions and closely related problems, among other achievements in the theory. Our main results are Theorems \ref{theorem section 1}, \ref{generalizedweaklycofinite}, \ref{cd123}, \ref{gcnumbers}, \ref{gpdfinite}, and \ref{dim+gpd=4}. In particular, they strengthen or extend results by several other authors, and most of them yield consequences which are new even in the ordinary theory of local cohomology, such as Corollaries \ref{cd=1-bis}, \ref{dim4-0}, and \ref{corollary four dimensional}; see also Proposition \ref{topcohomology} and Theorem \ref{criteria dim < 2} -- the latter, whose proof was inspired by that of Cuong, Goto, and Hoang \cite[Corollary 5.3]{CGH}, leads us in turn to an application (see Corollary \ref{theorem abelian category}) regarding the category of $I$-weakly cofinite modules of dimension at most 2 which, in the local setting, extends the case of dimension at most 1 carried out by Bahmanpour, Naghipour, and Sedghi \cite[Theorem 3.5]{BAH}. In the last part of the paper (see Section \ref{patho}), we provide a couple of intriguing examples about sets of associated primes of local cohomology modules, one of them illustrating that a non weakly Laskerian module $N$ over a polynomial ring may have $H^1_I(N)\neq 0$ with finitely many associated primes.

Our main techniques are based on the extensive use of certain spectral sequences (see Lemma \ref{ss}) and on exploring connections to suitable numerical invariants such as the generalized cohomological dimension (see Definition \ref{gcdim}) and the relative Gorenstein projective dimension (see Definition \ref{def-gpd}) of a given pair of modules $M, N$, where, in most of our results, $M$ is assumed to be finite and $N$ weakly Laskerian.

Throughout this paper, all rings are assumed to be commutative and Noetherian with multiplicative identity 1, and, as mentioned above, by a {\it finite} module we mean a finitely generated module.

\section{The first generalized local cohomology module}

We begin by investigating finiteness properties related to the first generalized local cohomology module $H_I^1(M, N)$ of certain $R$-modules $M$ and $N$; see Definition \ref{Herdef} below. Throughout the paper, $I$ is a proper ideal of the ring $R$, and as usual we set $V(I)=\{P\in {\rm Spec}\,R \mid I\subset P\}$. First, let us invoke a few concepts which play a central role in this note.

\begin{definition}\label{defWL}\rm (\cite{DAM}) Let $W$ be an $R$-module. 

\medskip

\noindent (i) $W$ is said to be \textit{weakly Laskerian} if the set $\Ass_RW/U$ is finite for every $R$-submodule $U$ of $W$;

\medskip

\noindent (ii) $W$ is said to be {\it $I$-weakly cofinite} if $\Supp_RW \subset V(I)$ and $\Ext^i_R (R/I, W)$ is weakly Laskerian for all $i\geq 0$.
\end{definition}

Below we collect some relevant facts on these notions (proofs can be found in \cite{DAM}). Part (iv), for instance, will be used in several proofs without explicit mention.

\begin{remark}\label{obss}\rm  (i) Any Noetherian module is weakly
Laskerian. If $\Supp_RW$ is finite, then $W$ is weakly
Laskerian; in particular, any Artinian $R$-module is weakly Laskerian. If $M$ is finite and $N$ is weakly Laskerian, then the $R$-modules $\Ext_R^i(M, N)$ and $\Tor^R_i(M, N)$ are weakly Laskerian for all $i \geq 0$.

\medskip

\noindent (ii) Assume that $\Supp_RW \subset V(I)$. If $W$ is weakly Laskerian, then $W$ is $I$-weakly cofinite; in particular, if $W$ is either finite or Artinian, then $W$ is $I$-weakly cofinite.

\medskip

\noindent (iii) If $W$ is $I$-weakly cofinite, then $\Ass_RW$ is finite.

\medskip

\noindent (iv) If any two of the $R$-modules in an exact sequence $\xymatrix@=1em{0\ar[r] & U\ar[r] & W\ar[r] & C\ar[r] & 0}$ (resp. $\xymatrix@=1em{U\ar[r] & W\ar[r] & C}$) are $I$-weakly cofinite (resp. weakly Laskerian) then so is the remaining module.

\end{remark}

\begin{definition}\rm (\cite{Herzog}, \cite{B})\label{Herdef} Given $i\geq 0$, the $i$-th {\it generalized local cohomology module} of the $R$-modules $M$, $N$ with respect to the ideal $I$ of $R$ is defined as 
$$H^i_I(M,N) = \displaystyle\varinjlim_{n} \Ext^i_R (M/I^n M,N).$$ For each $i$, the direct system is induced by the epimorphsms $\xymatrix@=1em{M/I^{n+1}M\ar[r] & M/I^nM.}$ Also note that, by letting $M=R$, this definition retrieves the ordinary local cohomology module $H^i_I(N)$ of $N$.
\end{definition}

Now suppose $N$ is a finite $R$-module. Besides $H^0_{I} (N)$, the first local cohomology module $H^1_{I} (N)$ is the only one for which the finiteness of the set of associated primes can be guaranteed without further assumptions.
This fact is well-known and a simple proof was given in \cite[Proposition 1.1(c)]{Marley}. When $N$ is replaced with a weakly Laskerian $R$-module (which is therefore not necessarily finite), this result remains true as shown in \cite[Corollary 2.7]{DAM}.


As expected, we may consider the same issue for $H^0_{I} (M,N)$ and $H^1_{I} (M,N)$. If again $M$ is finite and $N$ is weakly Laskerian, then the module $H^0_{I} (M,N)$ is $I$-weakly cofinite and consequently it has finitely many associated primes. The finiteness of the set $\Ass_R H^1_{I} (M,N)$ persists to hold as well and this can be deduced from \cite[Theorem 3.3]{AM}, as highlighted in Proposition \ref{prop Ass H^1(M,N)} below. In this section, our main goal is to establish the finiteness of $\Ass_R \Ext^1_R (R/I, H^1_{I} (M,N))$ in certain situations, the natural motivation for this problem being the equality $$\Ass_R \Hom_R (R/I, H^1_{I} (M,N)) = \Ass_R H^1_{I} (M,N),$$ which follows easily from \cite[Exercise 1.2.27]{BH} and the fact that $\Ass_R H^1_{I} (M,N)\subset V(I)$.


We record the following fact, which can be seen by taking $t=1$ in \cite[Theorem 3.3]{AM}.

\begin{proposition}\label{prop Ass H^1(M,N)}
 Suppose $M$ is finite and $N$ is weakly Laskerian. Then, $\Ass_RH^1_{I} (M,N)$ is finite.
\end{proposition}


Before moving to the main goal of this section, we recall some spectral sequence tools from \cite[Proposition 2.1]{FJMS} which will be used in most of our proofs.

\begin{lemma}\label{ss}
For $R$-modules $M$ and $N$, with $M$ finite, we have the following spectral sequences:
\begin{itemize}
    \item [\rm (i)] $E_2^{p,q}=H^p_I(\Ext^q_R(M,N))\Rightarrow_p H^{p+q}_I(M,N)$;
    \item [\rm (ii)] $E_2^{p,q}=\Ext^p_R(M,H^q_I(N))\Rightarrow_p H^{p+q}_I(M,N)$.
\end{itemize}
\end{lemma}

The next proposition extends \cite[Lemma 2.7]{MS}.

\begin{proposition}\label{fivetermsexactsequence}
Let $M,N$ be $R$-modules, with $M$ finite and $N$ weakly Laskerian. The following assertions are equivalent:

\begin{itemize}

\item [\rm(i)] $H^1_I(M,N)$ is $I$-weakly cofinite;

\item [\rm(ii)] $H^1_I(\Hom_R(M,N))$ is $I$-weakly cofinite;
    
\item [\rm(iii)] $\Hom_R(M,H^1_I(N))$ is $I$-weakly cofinite. 

\end{itemize}

\end{proposition}

\begin{proof}
Any of the spectral sequences $E$ in Lemma \ref{ss} is first quadrant and has a corner
$$\xymatrix@=1em{
E_2^{0,1}\ar[rrd]^{d_2} & E^{1,1}_2 & E^{2,1}_2 & \cdots
\\
E^{0,0}_2 & E^{1,0}_2 & E^{2,0} & \cdots
}$$
Note there are exact sequences
$$\xymatrix@=1em{0\ar[r] & E_2^{1,0}\ar[r] & H^1_I(M,N)\ar[r] & E_\infty^{0,1}\ar[r] & 0},$$
$$\xymatrix@=1em{0\ar[r] & E_\infty^{0,1}\ar[r] & E_2^{0,1}\ar[r]^{d_2} & E_2^{2,0}\ar[r] & E_\infty^{2,0}  \ar[r] & 0}.$$
Since $E_\infty^{2,0}\subset H^2_I(M,N)$, we may combine the above exact sequences into a five-term exact sequence $$\xymatrix@=1em{
& 0\ar[d]
\\
& E_2^{1,0}\ar[d]
\\
& H^1_I(M,N)\ar[d]\ar@{-->}[rd] & & & 0\ar[d]
\\
0\ar[r] & E_\infty^{0,1}\ar[r]\ar[d] & E_2^{0,1}\ar[r]^{d_2} & E_2^{2,0}\ar[r]\ar@{-->}[rd] & E_\infty^{2,0}\ar[r]\ar[d] & 0
\\
& 0 & & & H^2_I(M,N)
}$$
Now, for the equivalence (i)$\Leftrightarrow$(ii), we consider this sequence in light of the spectral sequence of Lemma \ref{ss}(i). We thus obtain a short exact sequence
$$\xymatrix@=1em{0\ar[r] & H^1_I(\Hom_R(M,N))\ar[r] & H^1_I(M,N)\ar[r] & C\ar[r] & 0},$$
where $C\subset H^0_I(\Ext^1_R(M,N))$.
Hence, $C$ is $I$-weakly cofinite, which enables us to deduce the equivalence between the $I$-weak cofiniteness of $H^1_{I} (M,N)$ and $H^1_{I} (\Hom_R (M,N))$.

Next, for the equivalence (i)$\Leftrightarrow$(iii), let us consider the spectral sequence given in Lemma \ref{ss}(ii) and its five-term exact sequence. From the exact sequence
$$\xymatrix@=1em{0\ar[r] & \Ext^1_R(M,H^0_I(N))\ar[r] & H^1_I(M,N)\ar[r] & \Hom_R(M,H^1_I(N))\ar[r] & C'\ar[r]& 0,}$$
and since $C'\subset\Ext^2_R(M,H^0_I(N))$ is $I$-weakly cofinite, we conclude that $H^1_I(M,N)$ is $I$-weakly cofinite if and only if  $\Hom_R(M,H^1_I(N))$ has the same property.
\end{proof}

\begin{definition}\rm The {\it $I$-depth} of an $R$-module $M$ is defined as $$\depth_{I} M = \inf \{i\geq 0 \mid H^i_{I} (M) \neq 0\}.$$ Clearly, if $M$ is finite and $M \neq IM$, then this definition agrees with the classical one, namely, the length of some (any) maximal $M$-sequence contained in $I$. If in addition $R$ is local and $I$ is its maximal ideal, then we write $\depth_{R} M$.
\end{definition}

\begin{remark}\label{depth>0remark}\rm
By the five-term exact sequence corresponding to the spectral sequence of Lemma \ref{ss}(ii) (see the proof of Proposition \ref{fivetermsexactsequence}) along with the condition $\depth_IN>0$ in place of the weak Laskerianess of $N$, we obtain in fact an isomorphism $$H^1_I(M,N)\cong\Hom_R(M,H^1_I(N)).$$ This extends \cite[Lemma 7.9]{Mel}.
\end{remark}

\begin{proposition}\label{proposition H^1}
Let $M, N$ be $R$-modules, with $M$ finite and $\depth_IN>0$. Assume that $\Tor_i^R(R/I,M)=0$ for all $i>0$ $($see Remark $\ref{tor0}$ below$)$. Then
$$\Ext^1_R(R/I,H^1_I(M,N))\subset\Ext^1_R(M/IM,H^1_I(N)).$$
\end{proposition}

\begin{proof}
By taking a projective resolution $P_\bullet$ of $R/I$ and an injective resolution $E^\bullet$ of $H^1_I(N)$, and by considering tensor-hom adjunction, we obtain an isomorphism of double complexes
$$\Hom_R(P_\bullet,\Hom_R(M,E^\bullet))\cong\Hom_R(P_\bullet\otimes_RM,E^\bullet)$$ from which we derive a spectral sequence
$$\Ext^i_R(R/I,\Ext^j_R(M,H^1_I(N)))\Rightarrow_i\Ext^{i+j}_R(M/IM,H^1_I(N)).$$ The result now follows by analyzing the corner of such a spectral sequence.
\end{proof}

Finally we come to the main result of this section, and its immediate consequence in the ordinary local cohomology case (which is probably known to experts).

\begin{theorem}\label{theorem section 1}
 Suppose $M$ is finite and $N$ is weakly Laskerian. If $\Tor^R_i (R/I,M)=0$ for all $i>0$ and $\depth_{I} N>0$, then
$\Ext^1_R (R/I, H^1_{I}(M,N))$ is weakly Laskerian. In particular, $\Ass_R\Ext^1_R (R/I, H^1_{I}(M,N))$ is finite.
\end{theorem}
\begin{proof}
First of all, by taking $s=1$ and $\mathcal{S}$ as the category of all weakly Laskerian $R$-modules in \cite[Lemma 2.1 and Lemma 2.2]{BAH}, we get the weak Laskerianess of $$\Ext^i_R (R/I,H^1_I(N)) \quad \mbox{for} \quad i=0, 1.$$ 
Therefore, by virtue of \cite[Lemma 2.2]{RM}, the $R$-module $\Ext^i_R (M/IM,H^1_I(N))$ is weakly Laskerian for $i=0,1$. Now, the assertion follows from Proposition \ref{proposition H^1}.
\end{proof}

\begin{corollary}\label{cor theorem section 1} If $N$ is weakly Laskerian and $\depth_{I} N>0$, then
$\Ext^1_R (R/I, H^1_{I}(N))$ is weakly Laskerian. In particular, $\Ass_R\Ext^1_R (R/I, H^1_{I}(N))$ is finite.
\end{corollary}

It seems natural to pose the following problem.

\begin{question}\rm Under the hypotheses of Theorem \ref{theorem section 1}, is it true that $H^1_{I}(M,N)$ is $I$-weakly cofinite? 


\end{question}

\begin{remark}\label{tor0}\rm Let us recall, for completeness, two instances where $\Tor_i^R(R/I,M)$ is known to vanish for all $i>0$, with $M$ finite, apart from the obvious case where $M$ is flat (e.g., $M=R$). First, if $I$ is generated by an $M$-sequence (this is easy and well-known); and second, if $R$ is local, $M$ is maximal Cohen-Macaulay (i.e., ${\rm depth}_RM={\rm dim}\,R$) and $I$ has finite projective dimension; see \cite[Lemma 2.2]{Y}. 

\end{remark}

\section{Weak cofiniteness of local cohomology modules}

The $I$-cofiniteness and $I$-weak cofiniteness of ordinary and generalized local cohomology modules have been extensively studied in the case of rings of small dimension. This section mainly aims to investigate the $I$-weak cofiniteness of $H^i_{I} (M,N)$ for small values of other numerical invariants attached to the given data, to wit, the generalized cohomological dimension and the relative Gorenstein projective dimension of the pair $M, N$. One of our goals is to furnish a version of \cite[Theorem 3.1]{DAM} (hence of \cite[Proposition 3.11]{Mel} as well) for generalized local cohomology modules.


First, before dividing this part into subsections, we provide a warm-up result.

\begin{proposition} Let $M, N$ be $R$-modules, with $M$  finite. Assume ${\rm Supp}_RM\cap {\rm Supp}_RN\subset V(I)$. If $N$ is weakly Laskerian $($resp. finite$)$, then  $H^i_I(M,N)$ is $I$-weakly cofinite $($resp. finite$)$ for all $i\geq 0$.
Consequently, ${\rm Ass}_RH^i_I(M,N)$ is finite for all $i\geq 0$.
\end{proposition}
\begin{proof} Applying \cite[Corollary 1.8]{FJMS}, we get isomorphisms
$$H^i_I(M,N)\cong {\rm Ext}_R^i(M, N) \quad \mbox{for \,all} \quad i\geq 0.$$ Note the case where $N$ is also finite is now obvious. Finally,
if $N$ is weakly Laskerian then by Remark \ref{obss}(i) the modules ${\rm Ext}_R^i(M, N)$ are  weakly Laskerian as well, and hence so is each $H^i_I(M,N)$. Because ${\rm Supp}_RH^i_I(M,N)\subset V(I)$, we obtain by Remark \ref{obss}(ii) that $H^i_I(M,N)$ is $I$-weakly cofinite for all $i\geq 0$, as needed. In particular, $H^i_I(M,N)$ has finitely many associated primes (see Remark \ref{obss}(iii)). \end{proof}


\subsection{Cohomologically injective pairs} We begin by establishing a weak cofiniteness criterion (Theorem \ref{generalizedweaklycofinite} below) which generalizes \cite[Theorem 3.1]{DAM} and is particularly useful when the module $H^i_{I} (M,N)$ vanishes for many indices. To this end, we introduce a technical definition which is tailored to make possible extending some results from the ordinary to the generalized local cohomology setting.

\begin{definition}\label{cohomolinject}\rm Let $M, N$ be $R$-modules, with $M$  finite. We say that the pair $M, N$ is {\it $I$-cohomologically injective} if the 0-th ordinary local cohomology module $H^0_I({\rm Hom}_R(M,E^i))$ is injective for each $E^i$ appearing in some injective resolution $E^\bullet$ of the $R$-module $N$.
\end{definition}

\begin{remark}\label{ex-pairs}\rm Some words about Definition \ref{cohomolinject} are in order. 

\smallskip

\noindent (i) We do not know whether the above definition is independent of the choice of $E^\bullet$ (we believe it is). However, we claim that if $M, N$ is $I$-cohomologically injective, then the injective resolution $E^\bullet$ of $N$ can be taken minimal. Indeed, if $\widetilde{E}^\bullet$ is a minimal injective resolution of $N$, then there is an exact complex of $R$-modules $C^{\bullet}$ such that $E^\bullet = \widetilde{E}^\bullet \oplus C^{\bullet}$. Now, the claim follows by the isomorphism $$H^0_I({\rm Hom}_R(M,E^i)) \cong H^0_I({\rm Hom}_R(M,\widetilde{E}^i)) \oplus H^0_I({\rm Hom}_R(M,C^i)),$$ which ensures that $H^0_I({\rm Hom}_R(M,\widetilde{E}^i))$ must be injective if so is $H^0_I({\rm Hom}_R(M, E^i))$.



\medskip

\noindent (ii) Finite flat $R$-modules $F$ are such that the module $H^0_I({\rm Hom}_R(F, E)) \cong H^0_I(F, E)$ is injective whenever $E$ is injective -- in particular, $F, N$ is $I$-cohomologically injective for all $R$-modules $N$. Indeed, note first $H^0_I(F, E)\cong\Hom_R(F, H^0_I(E))$ (this is well-known but follows alternatively from the spectral sequence given in Lemma \ref{ss}(ii)). Thus, the functor $$\Hom_R(-, H^0_I({\rm Hom}_R(F, E)))\cong\Hom_R(- \otimes_RF,H^0_I(E))$$ is exact as $H^0_I(E)$ is injective. Reciprocally, by the same argument, if $M$ is a finite $R$-module such that $H^0_I({\rm Hom}_R(M,E))$ is injective and $\Hom_R(-, H^0_I(E))$ is a faithful functor for some injective $R$-module $E$, then $M$ must be flat. In particular, in case $(R, {\bf m})$ is local, the module $H^0_{\bf m}(M,E_R(R/{\bf m}))$ is injective if and only if $M$ is free, where $E_R(R/{\bf m})$ stands for the injective hull of the residue field $R/{\bf m}$. 


\end{remark}

Recall that a non-zero finite $R$-module $N$ is a \textit{Gorenstein module} if its Cousin complex $C^{\bullet}(N)$ is a minimal injective resolution of $N$. Details can be found in \cite{Sharp}, \cite{Sharp2}.

\begin{proposition}
Let $R$ be a Gorenstein local ring and $M$ a finite $R$-module. Then, the pair $M, R$ is $I$-cohomologically injective, for some ideal $I$ of $R$, if and only if $M$ is free.
More generally, the same statement holds if we replace $R$ with a Gorenstein module $N$.
\end{proposition}
\begin{proof} First, the pair  $M, R$ is easily seen to be $I$-cohomologically injective if $M$ is free. For the converse,  since the given local ring $(R, {\bf m})$ is Gorenstein, the module $E_R (R/\textbf{m})$ is part of a minimal injective resolution of $R$. Note
$$E_R (R/\textbf{m})=H^0_{\bf m} (E_R (R/\textbf{m})) \subset H^0_{I} (E_R (R/\textbf{m}))\subset E_R (R/\textbf{m}),$$ and so these are all equalities. Now, using Remark \ref{ex-pairs}(ii), we derive that the module $H^0_I({\rm Hom}_R(M, E_R (R/\textbf{m}))) \cong \Hom_R (M, H^0_{\bf m} (E_R (R/\textbf{m})))\cong H^0_{\bf m}(M,E_R(R/{\bf m}))$ is injective if and only if $M$ is free. Thus, $M$ is free if $M, R$ is $I$-cohomologically injective, as needed.

Finally, suppose that the $R$-module $N$ is Gorenstein and that the pair $M,N$ is $I$-cohomologically injective. By \cite[Proposition 3.5]{Sharp2}, if $\dim N=n$ then the $n$-th term in $C^{\bullet}(N)$ is given by $C^n(N) = E_R (R/ \textbf{m})^{\mu^n(\textbf{m}, N)}$, where $\mu^n(\textbf{m}, N)$ is the $n$-th Bass number of $N$ with respect to $\textbf{m}$.
Hence, the hypothesis forces the module $$H^0_I({\rm Hom}_R(M, C^n(N)))\cong \Hom_R (M, H^0_I (E_R (R/\textbf{m})))^{\mu^n(\textbf{m}, N)}$$ to be injective, so that $\Hom_R (M, H^0_I (E_R (R/\textbf{m})))$ must be injective as well, which as seen above is equivalent to $M$ being free.
\end{proof}

Now, motivated by the above proposition and Remark \ref{ex-pairs}(ii), we can naturally consider the following complementary question.


\begin{question}\rm
Let $R$ be a local ring. If $M$ is finite with $\pd_RM<\infty$, is it true that 
the pair $M, N$ is $I$-cohomologically injective for every (or at least for some) non-Gorenstein $R$-module $N$? 


\end{question}


The theorem below is the main result of this subsection. It will be also a crucial tool in other parts of the paper such as in the proof of Proposition \ref{coh formula new} and Theorem \ref{gpdfinite}.

\begin{theorem}\label{generalizedweaklycofinite} Suppose $M$ is finite and $N$ is weakly Laskerian. Assume that the pair $M, N$ is $I$-cohomologically injective. If there is an integer $s\geq 0$ such that $H^i_I(M,N)$ is $I$-weakly cofinite for all $i\neq s$, then  $H^s_I(M,N)$ is $I$-weakly cofinite.

\end{theorem}
\begin{proof}
By taking a projective resolution $P_\bullet$ of $R/I$ as well as an injective resolution $E^\bullet$ of $N$ as in Definition \ref{cohomolinject}, the double complex
$$\Hom_R(P_\bullet,H^0_I({\rm Hom}_R(M,E^\bullet)))\cong \Hom_R(P_\bullet,H^0_I(M,E^\bullet))$$
yields a spectral sequence $$E_2^{p,q}=\Ext^p_R(R/I,H^q_I(M,N))\Rightarrow_p\Ext^{p+q}_R(M/IM,N).$$


We now argue in the same fashion as in \cite[Proposition 2.5]{MV}. By assumption, the $R$-modules $\Ext^i_R(M/IM,N)$ and $E_r^{p,q}$ are weakly Laskerian for all $r\geq 2$ and $p, q, i\geq 0 $, with $q\neq s$. In particular, as $E_\infty^{p,q}$ is a subquotient of $\Ext^{p+q}_R(M/IM,N)$, it must be weakly Laskerian as well. Notice also that given $p,q\geq0$, there is $r$ large enough such that $E_r^{p,q}$ is weakly Laskerian.

Now, consider the differentials $d_r^{p,q}:E_r^{p,q}\rightarrow E_r^{p+r,q-r+1}$ for $r\geq2$. For each $q\geq0$, either $E_r^{p,q}$ or $E_r^{p+r,s-q+1}$ is weakly Laskerian and then so is $\im d_r^{p,q}$. If we let $r\geq2$ be large enough so that $E_{r+1}^{p,s}$ is weakly Laskerian, then from the exact sequence
$$\xymatrix@=1em{
0\ar[r] & \im d_r^{p-r,s+r-1}\ar[r] & \ker d_r^{p,s}\ar[r] & E_{r+1}^{p,s}\ar[r] & 0
}$$ we deduce that $\ker d_r^{p,s}$ is weakly Laskerian. Therefore, by using the exact sequence
$$\xymatrix@=1em{
0\ar[r] & \ker d_r^{p,s}\ar[r] & E_r^{p,s}\ar[r] & \im d_r^{p,s}\ar[r] & 0
}$$ we deduce the weak Laskerianess of $E_r^{p,s}$. By applying this argument successively, we conclude that $E_2^{p,s}$ is weakly Laskerian.
\end{proof}

\begin{remark}\label{cofiniteremark}\rm
If in Theorem \ref{generalizedweaklycofinite} we assume that $N$ is finite, then we can replace {\it $I$-weakly cofinite} with {\it $I$-cofinite} in the statement.
\end{remark}

\begin{corollary}
Suppose the hypotheses of Theorem \ref{generalizedweaklycofinite}. If there is an integer $s\geq 0$ such that $H^i_{I} (M,N)=0$ for all $i \neq s$, then $H^s_{I} (M,N)$ is $I$-weakly cofinite.
\end{corollary}

We then record \cite[Theorem 3.1]{DAM} by taking $M=R$ in Theorem \ref{generalizedweaklycofinite}.

\begin{corollary}
If $N$ is weakly Laskerian and there is an integer $s\geq0$ such that $H^i_I(N)$ is $I$-weakly cofinite for all $i\neq s$, then $H^s_I(N)$ is $I$-weakly cofinite.
\end{corollary}


Another immediate byproduct of Theorem \ref{generalizedweaklycofinite} is the following helpful result which will be used several times in this note.

\begin{corollary}\label{ss+1}
Suppose $M$ is finite and $N$ is weakly Laskerian. Assume that the pair $M, N$ is $I$-cohomologically injective. If there are distinct integers $s, k\geq 0$ such that $H^i_I(M,N)$ is $I$-weakly cofinite for all $i\neq s, k$, then $H^s_I(M,N)$ is $I$-weakly cofinite if and only if $H^{k}_I(M,N)$ is $I$-weakly cofinite.
\end{corollary}


\subsection{Cohomological dimension and weak cofiniteness}
In this part, we apply Theorem \ref{generalizedweaklycofinite} in order to obtain results concerning the $I$-weak cofiniteness of $H^i_{I} (M,N)$, particularly when $\cd_I(M,N) \leq 3$, where the number $\cd_I(M,N)$ is as defined below.


\begin{definition}\label{gcdim}\rm (\cite{JN}) Let $M, N$ be $R$-modules. The ({\it generalized}) {\it cohomological dimension} of $M, N$ with respect to the ideal $I$ is defined as $$\cd_{I} (M,N)= \sup \{i \geq 0 \mid H^i_{I} (M,N) \neq 0\}.
$$ Note that, if $M=R$, we recover the standard definition of the cohomological dimension $\cd_{I}N$ of $N$ with respect to $I$. Moreover, it is usual to define the cohomological dimension of $I$, denoted $\cd I$, as the smallest integer $m\geq 0$ such that $H^j_{I}(T)=0$ for all $R$-modules $T$ and  all integers $j > m$. By \cite[Proposition 6.1.11]{Brodman}, we have an equality $$\cd I=\cd_{I}R.$$ This invariant will be used in Corollaries \ref{cd=1} and \ref{cd=1-bis}. It is worth recalling that, under suitable hypotheses on a given local ring $(R, {\bf m})$, the number $\cd I$ is closely related to the punctured scheme 
${\rm Spec}\,R/I\,\setminus \{{\bf m}\}$ being connected (see \cite{H0} and \cite{Huneke-L}).


\end{definition}

Now we observe the following theorem.

\begin{theorem}\label{cd123}
Suppose $M$ is finite and $N$ is weakly Laskerian. Suppose the pair $M, N$ is $I$-cohomologically injective. Then the following assertions hold:
\begin{itemize}
    \item [\rm(i)] If $\cd_I(M,N)=1$, then $H^1_I(M,N)$ is $I$-weakly cofinite;
    
    \item [\rm(ii)] If $\cd_I(M,N)=2$, then $H^1_I(M,N)$ is $I$-weakly cofinite if and only if  $H^2_I(M,N)$ is $I$-weakly cofinite;
    
    \item [\rm(iii)] Suppose that $\cd_I(M,N)=3$ and that $i,j,k$ are three different indices in $\{1,2,3\}$. If $H^i_I(M,N)$ is $I$-weakly cofinite, then $H^j_I(M,N)$ is $I$-weakly cofinite if and only if  $H^k_I(M,N)$ is $I$-weakly cofinite.
\end{itemize}
\end{theorem}

\begin{proof}
Consider first the fact that the module $H^0_I(M,N)\cong H^0_I({\rm Hom}_R(M,N))$ embeds into ${\rm Hom}_R(M,N)$, which is known to be $I$-weakly cofinite because $M$ is finite and $N$ is weakly Laskerian. Thus, $H^0_I(M,N)$ is $I$-weakly cofinite as well. Now,  Corollary \ref{ss+1} applies.
\end{proof}

\begin{remark}\label{H^1remark}\rm
Using Proposition \ref{fivetermsexactsequence}, and in the setting of Theorem \ref{cd123}(ii), we deduce that the $I$-weak cofiniteness of any of the modules $$\Hom_R(M,H^1_I(N)),\  H^1_I(\Hom_R(M,N)),\  H^1_I(M,N),\ H^2_I(M,N)$$ yields the same property for all the others. It is worth mentioning that, for the  ordinary local cohomology modules $H^i_I(N)$ of the weakly Laskerian $R$-module $N$, the $I$-weak cofiniteness is known to hold in some situations, for instance (supposing $R$ is local) if ${\rm dim}\,R\leq 3$ (see \cite[Corollary 3.4]{DAM}), and if $\dim M\leq 3$ or $\dim N\leq 3$ (see \cite[Corollary 5.3]{CGH}).
\end{remark}


Letting $M=R$ and taking Remark \ref{ex-pairs}(ii) into account, we can immediately record Theorem \ref{cd123} for ordinary local cohomology as the following corollary.

\begin{corollary}\label{cdordinary}
Suppose $N$ is weakly Laskerian. The following assertions hold:
\begin{itemize}
    \item [\rm(i)] If $\cd_IN=1$, then $H^1_I(N)$ is $I$-weakly cofinite;
    
    \item [\rm(ii)] If $\cd_IN=2$, then $H^1_I(N)$ is $I$-weakly cofinite if and only if  $H^2_I(N)$ is $I$-weakly cofinite;
    
    \item [\rm(iii)] Suppose that $\cd_IN=3$ and that $i,j,k$ are three different indices in $\{1,2,3\}$. If $H^i_I(N)$ is $I$-weakly cofinite, then $H^j_I(N)$ is $I$-weakly cofinite if and only if $H^k_I(N)$ is $I$-weakly cofinite.
\end{itemize}
\end{corollary}

	

Next, before proving further results, we introduce some notations for convenience.

\begin{notation}\rm  Let $M,N$ be two $R$-modules. We write
$$g(M, N)=\inf_{j\geq 0}\{\depth_I\Ext^j_R(M,N)\}\quad\mbox{and}\quad c(M, N)=\sup_{j\geq 0}\{\cd_I \Ext^j_R(M,N)\}.$$
\end{notation}

There is, therefore, a deviation  $\delta(M, N)=c(M, N)-g(M, N)\geq 0$, and notice that due to the spectral sequence of Lemma \ref{ss}(i) we must have $H^i_I(M,N)=0$ if either $i<g(M, N)$ or $i>c(M, N)$. The next theorem discusses the $I$-weak cofiniteness of $H^i_{I} (M,N)$ in the cases of deviation $0$ and $1$.


\begin{theorem}\label{gcnumbers}
Suppose $M$ is finite and $N$ is weakly Laskerian.  Assume that any one of the following situations holds:
\begin{itemize}
    \item [\rm(i)] $\delta(M, N)=0$;
    \item [\rm(ii)] $\delta(M, N)=1$ and either $H^{g(M, N)}_I(\Ext^j_R(M,N))$ or $H^{c(M, N)}_I(\Ext^j_R(M,N))$ is $I$-weakly cofinite, for each $j\geq 0$.
\end{itemize} Then, $H^i_I(M,N)$ is $I$-weakly cofinite for all $i\geq 0$.
\end{theorem}


\begin{proof}
Applying Corollary \ref{ss+1} to the pair $R, \Ext^j_R(M,N)$ for each $j$, we deduce that $H^i_I(\Ext^j_R(M,N))$ is $I$-weakly cofinite for all $i\geq 0$, in both situations. Now, consider the spectral sequence of Lemma \ref{ss}(i),
$$E_2^{i,j}=H^i_I(\Ext^j_R(M,N))\Rightarrow_i H^{i+j}_I(M,N).$$
Set for simplicity $g=g(M, N)$ and $c=c(M, N)$. Note that $E_2^{i,j}=0$ whenever  $i<g$ or $i>c$. It follows that, if $g=c$, then $E$ has only one vertical line $E_2^{c,j}$ with possible non-zero modules and therefore, by convergence,
$$H^c_I(\Ext^{i-c}_R(M,N))\cong H^i_I(M,N) \quad \mbox{for \,all} \quad i\geq 0,$$ where of course $\Ext^{k}_R(M,N)=0$ whenever $k<0$. This yields the result in this case.

Now suppose $g+1=c$. Then, $E$ has only two vertical lines with possible non-zero modules. Since its differentials in the second page have bidegree $(2,-1)$, i.e., $E_2^{i,j}\rightarrow E_2^{i+2,j-1}$ for any $i,j$, we  deduce that $E_2=E_\infty$. By virtue of convergence once again, there is for each $i\geq0$ a short exact sequence
$$\xymatrix@=1em{
0\ar[r] & H^c_I(\Ext^{i-c}_R(M,N))\ar[r] & H^i_I(M,N)\ar[r] & H^g_I(\Ext^{i-g}_R(M,N))\ar[r] & 0,
}$$ from which we derive the $I$-weak cofiniteness of $H^i_{I} (M,N)$.
\end{proof}

Now it seems plausible to propose the following question.

\begin{question}\rm Let $M$ be finite and $N$ weakly Laskerian, with $\delta (M, N)=\delta \geq 2$. Assume that, for each $j\geq 0$, at least $\delta$ of the modules $$H^{g(M, N)+k}_I(\Ext^j_R(M,N)), \quad k=0, \ldots, \delta,$$ are $I$-weakly cofinite. Is it true that $H^i_I(M,N)$ is  $I$-weakly cofinite for all $i\geq 0$?

\end{question}

Theorem \ref{gcnumbers} puts us in a position to establish the $I$-weak cofiniteness of the module $H^i_{I} (M,N)$ for all $i\geq 0$ whenever $\cd I=1$. This improves
\cite[Theorem 2.2]{MS}.

\begin{corollary}\label{cd=1}
Suppose $M$ is finite and $N$ is weakly Laskerian. If $\cd I =1$, then the $R$-module $H^i_I(M,N)$ is $I$-weakly cofinite for all $i\geq0$. In particular, $\Ass_RH^i_{I} (M,N)$ is finite for all $i\geq 0$.
\end{corollary}
\begin{proof} Since $\cd I =1$ we obtain in particular that, for each $j$, the module $H^i_I(\Ext^j_R(M, N))$ vanishes for all $i\geq 2$. Moreover, $0\leq g(M, N)\leq c(M, N)\leq 1$, and recall that $H^0_I(\Ext^j_R(M, N))\subset \Ext^j_R(M, N)$ is $I$-weakly cofinite. 
Now the result follows by Theorem \ref{gcnumbers}.
\end{proof}

Another consequence is the following extension of \cite[Theorem 2.10]{MS}.

\begin{corollary}\label{cd=1-bis}
Suppose $M$ is finite and $N$ is weakly Laskerian. If $\cd I =1$, then the $R$-module $\Ext^i_R(M,H^j_I(N))$ is $I$-weakly cofinite for all $i,j\geq0$.
\end{corollary}
\begin{proof}
Consider the spectral sequence of Lemma \ref{ss}(ii),
$$E_2^{p,q}=\Ext^p_R(M,H^q_I(N))\Rightarrow_p H^{p+q}_I(M,N).$$
Note it has only two horizontal lines with possible non-zero modules. In particular, for each $i\geq 0$ there is a short exact sequence
$$\xymatrix@=1em{
0\ar[r] & E^{i,0}_\infty\ar[r] & H^i_I(M,N)\ar[r] & E_\infty^{i-1,1}\ar[r] & 0,
}$$ where $E_\infty^{i,0}$ and $E_\infty^{i-1,1}$ are subquotients of the modules $E_2^{i,0}=\Ext^{i}_R(M,H^0_I(N))$ and $E_2^{i-1, 1}=\Ext^{i-1}_R(M,H^1_I(N))$, respectively. Now recall $N$ is weakly Laskerian and hence so is $E^{i,0}_2$. Thus, $E^{i,0}_\infty$ is weakly Laskerian for all $i\geq0$. By Corollary \ref{cd=1}, $H^i_I(M,N)$ is $I$-weakly cofinite and therefore $E^{i-1,1}_\infty$ is $I$-weakly cofinite as well. Finally, given $i\geq0$, and noticing that $E_\infty^{i,1}\cong\ker(E_2^{i,1}\rightarrow E_2^{i+2,0})$, the result follows from the exact sequence $$\xymatrix@=1em{0\ar[r] & E_\infty^{i,1}\ar[r] & E_2^{i,1}\ar[r] & E_2^{i+2,0}.}$$\end{proof}

\begin{remark}\label{non-sharp-cd}\rm Suppose $R$ is local and $M, N$ are both finite, with $M$ having finite projective dimension over $R$, say equal to $p$. Set $n={\rm cd}_IN$. Then, by \cite[Corollary 2.3]{FJMS}, we have $H_I^i(M, N)=0$ for all $i>p+n$, while, for $i=p+n$,
$$H_I^{p+n}(M, N)\cong \Ext^p_R(M,H^n_I(N))\cong \Ext^p_R(M, R)\otimes_RH_I^n(N).$$ In particular, if $\cd I=1$ (which coincides with $\cd_IR$, by \cite[Proposition 6.1.11]{Brodman}) then $H_I^i(M, R)=0$ for all $i>p+1$ and $H_I^{p+1}(M, R)\cong \Ext^p_R(M,H^1_I(R))\cong \Ext^p_R(M, R)\otimes_RH_I^1(R)$. Now it is worth illustrating that $H_I^{p+1}(M, R)$ can vanish, i.e., the inequality $\cd_I(M, R)\leq p+1$ can be strict. Indeed, assume that $(R, {\bf m})$ is a Gorenstein local ring of dimension 1, and that $p=1$ (e.g., $M=R/(x)$ for some regular element $x\in {\bf m}$). Then, $\cd {\bf m}=1$ and $H_{\bf m}^1(R)\cong E_R(R/{\bf m})$, which is an injective $R$-module and therefore $\Ext^1_R(M,H^1_{\bf m}(R))=0$, so that $H_{\bf m}^{2}(M, R)=0$.
 
\end{remark}

We close this subsection presenting conditions under which the bound ${\rm cd}_I (M,N) \leq p+n$ is sharp. We remark that the proposition remains valid if we replace $I$-weak cofiniteness with $I$-cofiniteness.

\begin{proposition}\label{coh formula new}
Let $R$ be a local ring and suppose $M,N$ is a pair of $I$-cohomologically injective finite $R$-modules with ${\rm pd}_R M = p < \infty$, and consider an integer $n\geq \cd_IN$.
If there is an integer $k\in  \{1, \ldots, p+n-1\}$ such that $H^i_I (M,N)$ is $I$-weakly cofinite for all $i \in \{1, \ldots, p+n-1\} \setminus \{k\}$ and $H^k_I (M,N)$ is not $I$-weakly cofinite, then $n= \cd_IN$ and $${\rm cd}_I (M,N)= p+n.$$
\end{proposition}
\begin{proof} Since ${\rm cd}_I (M,N) \leq p+\cd_IN \leq p+n$, it suffices to show that $H_I^{p+n}(M,N)\neq 0$. Suppose otherwise.  Then Theorem \ref{generalizedweaklycofinite} implies that $H^k_I (M,N)$ is $I$-weakly cofinite, which contradicts the assumption. 
\end{proof}


\subsection{Gorenstein projective dimension and weak cofiniteness} We begin by considering suitable numbers (attached to a given pair of modules) which will play a crucial role in this subsection.

\begin{definition}\label{def-gpd}\rm (\cite{DAH}) Let $M,N$ be two finite $R$-modules. The {\it Gorenstein projective dimension of $M$ relative to $N$} is given by
$$\Gpd_NM=\sup \{i \geq 0 \mid {\rm Ext}^i_R(M,N) \neq 0\}.$$
\end{definition}

Here it is worth recalling that if $M, N$ are  finite $R$-modules and $\pd_RM<\infty$ then $\Gpd_NM=\pd_RM$ (see \cite[Lemma 2.2(iv)]{DAH}).

\begin{notation}\label{delta}\rm (\cite{DAH}, \cite{DAH1}) Let $M,N$ be two finite $R$-modules. We set
$$\rho(M, N) = \dim M \otimes_R N + \Gpd_NM.$$
\end{notation}

In \cite[Theorem 2.7]{DAH1}, the $I$-weak cofiniteness of the module $H^i_{I} (M,N)$ was established when $\rho(M, N) \leq 2$. Our main goal in this part is to address this problem in the cases $\rho(M, N)=3$ and $\rho(M, N)=4$.

\begin{theorem}\label{gpdfinite} Suppose $R$ is local and $M, N$ is a pair of $I$-cohomologically injective finite $R$-modules. The following assertions hold:
\begin{itemize}
    \item [\rm(i)] If $\rho(M, N)\leq3$, then $H^i_I(M,N)$ is $I$-weakly cofinite for all $i\geq0$. In particular, $\Ass_R H^i_{I} (M,N)$ is finite for all $i\geq 0$;
    \item [\rm(ii)] If $\rho(M, N)=4$, then $H^1_I(M,N)$ is $I$-weakly cofinite if and only if $H^2_I(M,N)$ is $I$-weakly cofinite. Moreover, $\Ass_RH^i_{I} (M,N)$ is finite for all $i \neq 2$;
    \item[\rm(iii)] If $\rho(M, N)=4$, $R$ is local, and $\dim \Hom_R (M,N) \leq 3$, then $H^i_I(M,N)$ is $I$-weakly cofinite for all $i\geq0$. In particular, $\Ass_R H^i_I(M,N)$ is finite for all $i\geq 0$.
\end{itemize}
\end{theorem}

\begin{proof}
 (i) Set $\rho = \rho(M, N)$. Note that $\cd_IM\otimes_RN\leq \dim M \otimes_R N$, which gives 
    $$\Gpd_NM+\cd_IM\otimes_RN \leq \rho \leq 3.$$ Thus, by \cite[Theorem 2.5]{DAH}, we get $H^i_I(M,N)=0$ for all $i>3$ and, in addition, the $R$-modules $H^{\rho}_I(M,N)$ and $H^{{\rho}-1}_I(M,N)$ have finite support. As already mentioned, the case $\rho \leq 2$ is well-known by \cite[Theorem 2.7]{DAH1}, so we can assume $\rho =3$. Now, \cite[Lemma 2.2(iii)]{DAM} ensures that the $R$-modules $H^{3}_I(M,N)$ and $H^{2}_I(M,N)$ are weakly Laskerian, hence $I$-weakly cofinite. Because $H^0_I(M,N)$ is always $I$-weakly cofinite, our Theorem \ref{generalizedweaklycofinite} forces $H^1_I(M,N)$ to have the same property.
    This proves the assertion.
    
    \medskip
    
(ii) In analogy to the proof of part (i) above, we obtain that $H^i_I(M,N)=0$ for all $i>4$ and that both $H^4_I(M,N)$ and $H^3_I(M,N)$ are $I$-weakly cofinite. Now, using again that $H^0_I(M,N)$ is $I$-weakly cofinite, the result follows by Corollary \ref{ss+1} along with Proposition \ref{prop Ass H^1(M,N)}.

    \medskip
    
    (iii) Applying \cite[Corollary 5.3]{CGH} to the pair $R, \Hom_R (M,N)$, we deduce that the $R$-module $H^i_I(\Hom_R(M,N))$ is $I$-weakly cofinite for all $i\geq0$. By Proposition \ref{fivetermsexactsequence} we get that $H^1_I(M,N)$ is $I$-weakly cofinite. Now the result follows from  (ii) and its proof.
\end{proof}


	
	 
	 \begin{remark}\rm (i) Concerning Theorem \ref{gpdfinite}(iii), it might be of interest to understand the limit  situation where both $\rho(M, N)$ and $\dim \Hom_R (M,N)$ are equal to $4$ (cf.\,also Theorem \ref{dim+gpd=4} and its proof).
	 
	 \medskip
	 
	  \noindent (ii) By the proof of the theorem we can see that, if $R$ is local and $M, N$ are finite $R$-modules with $\rho(M, N)<\infty$, then $H^i_I(M, N)=0$ for all $i>\rho(M, N)$; this follows alternatively from \cite[Corollary 2.2]{FJMS}, which in addition gives an isomorphism $$H^{\rho(M, N)}_I(M, N)\cong H_I^{\dim M\otimes_RN}({\rm Ext}_R^{{\rm Gpd}_NM}(M, N)),$$ and notice {\it en passant} that, by \cite[Theorem 2.5]{DAH}, these modules must be Artinian. Now we observe that  $H^{\rho(M, N)}_I(M, N)=0$ if $\dim {\rm Ext}_R^{{\rm Gpd}_NM}(M, N)<\dim M\otimes_RN$ (recall that $\dim  {\rm Ext}_R^i(M, N)\leq \dim M\otimes_RN$ for all $i$). 
	 
	 \end{remark}
	 
	 By Theorem \ref{gpdfinite}(ii) and its proof, we immediately derive the following new fact in ordinary local cohomology theory. The result will be further complemented in Corollary \ref{corollary four dimensional}.
	 
	 \begin{corollary}\label{dim4-0}
    Suppose $R$ is local and $N$ is finite with $\dim N=4$. The  following conditions are equivalent:
    \begin{itemize}
        \item[\rm (i)] $H^i_I (N)$ is $I$-weakly cofinite for all $i$;
        \item[\rm (ii)] $H^1_I (N)$ is $I$-weakly cofinite;
        \item[\rm (iii)] $H^2_I (N)$ is $I$-weakly cofinite.
    \end{itemize}
\end{corollary}


	 Next, we record an auxiliary lemma.
	 

	\begin{lemma}\label{corollary ker f coker f} Let $f$ be an $R$-linear map between two $R$-modules. If $\Ext^i_R (R/I, \ker f)$ and $\Ext^i_R (R/I, \coker f)$ are weakly Laskerian for all $i\geq 0$, then the modules $\ker \Ext^i_R (R/I,f)$ and $\coker \Ext^i_R (R/I,f)$ are also weakly Laskerian for all $i\geq 0$.
	\end{lemma}
	\begin{proof} The category of all weakly Laskerian $R$-modules is known to be a Serre subcategory (see \cite[Lemma 2.2]{DAM}).
	Now the lemma is readily obtained by letting $\mathcal{S}$ be such a category in \cite[Corollary 3.2]{Mel}.
	\end{proof}

	The result below provides a criterion for $I$-weak cofiniteness analogous to the criterion for $I$-cofiniteness discussed in \cite[Corollary 3.4]{Mel}. It will be a crucial tool in the proof of Theorem \ref{dim+gpd=4}.
	
\begin{corollary}\label{corollary melkersson weakly cofinite}
	Let $H$ be an $R$-module with $\Supp_RH\subset V(I)$. Suppose there exists $x \in I$ such that $0:_H x$ and $H/xH$ are $I$-weakly cofinite. Then, $H$ is $I$-weakly cofinite.
	\end{corollary}
\begin{proof}
	Set $f=x 1_H\in {\rm Hom}_R(H, H)$. By assumption, the $R$-modules $\Ext^i_R(R/I, \ker f)$ and $\Ext^i_R (R/I, \coker f)$ are weakly Laskerian for all $i\geq 0$. Since $x\in I$ and the map $\Ext^i_R (R/I,f)$ is multiplication by $x$ on $\Ext^i_R (R/I, H)$, it must be the zero map, i.e., $$\ker \Ext^i_R (R/I,f) = \Ext^i_R (R/I, H).$$ Now, using Lemma \ref{corollary ker f coker f}, we deduce that $\Ext^i_R (R/I, H)$ is weakly Laskerian for all $i$. Since in addition $\Supp_RH\subset V(I)$, we conclude that $H$ is $I$-weakly cofinite.
	\end{proof}


We observe the following result.
	
\begin{proposition}\label{topcohomology}	
If $N$ is weakly Laskerian of dimension $d$, then $\Supp_RH^d_I (N)$ is finite. Consequently, $H^d_I (N)$ is weakly Laskerian. In particular, $H^d_I(N)$ is $I$-weakly cofinite.
	\end{proposition}
	\begin{proof}
In view of \cite[Theorem 3.3]{Bah0}, there exists a finite submodule $N_1$ of $N$ such that $\Supp_RN/N_1$ is finite. The short exact sequence
$\xymatrix@=1em{
0\ar[r] & N_1\ar[r] & N\ar[r] & N/N_1\ar[r] & 0
}$ induces an exact sequence
$$\xymatrix@=1em{
H^d_I (N_1)\ar[r] & H^d_I (N)\ar[r] & H^d_I (N/N_1)\ar[r] & 0.
}$$ Note that $\Supp_RH^d_I (N/N_1)$ is finite.
On the other hand, $\dim N_1 \leq d$. If $\dim N_1 < d$, then $H^d_I (N_1)=0$ and the claim follows. Otherwise, by using \cite[Exercise 7.1.7]{Brodman} we deduce that $H^d_I (N_1)$ is Artinian and hence $\Supp_RH^d_I (N_1)$ is finite. Now the finiteness of $\Supp_RH^d_I (N)$ is an immediate consequence of the above exact sequence. By Remark \ref{obss}(i), the module $H^d_I (N)$ is weakly Laskerian and hence, as clearly $\Supp_RH^d_I (N)\subset V(I)$, it must be $I$-weakly cofinite (see Remark \ref{obss}(ii)). \end{proof}
	
Now we  generalize \cite[Corollary 5.2 and Corollary 5.1]{CGH} respectively as follows.

\begin{corollary} Suppose $N$ is weakly Laskerian. Then the following assertions hold:
	\begin{itemize} 
	\item [\rm (i)] If $\dim  N\leq2$, then $H^i_I(N)$ is $I$-weakly cofinite for all $i\geq0$;
	
	\item [\rm (ii)] If $\dim  R\leq2$ and $M$ is finite, then $H^i_I(M,N)$ is $I$-weakly cofinite for all $i\geq0$.
	\end{itemize}
	\end{corollary}
	\begin{proof} Part (i) follows easily by Corollary \ref{cdordinary} and Proposition \ref{topcohomology}. To prove (ii), consider the spectral sequence $E$ of Lemma \ref{ss}(i). Using item (i) above, $E_2^{p,q}$ is $I$-weakly cofinite for all $p,q\geq0$. From the isomorphism $E_2^{1,i}\cong E_\infty^{1,i}$ and the exact sequence
	$$\xymatrix@=1em{
	0\ar[r] & E_\infty^{0,j}\ar[r] & E_2^{0,j}\ar[r] & E_2^{2,j-1}\ar[r] & E_\infty^{2,j-1}\ar[r] & 0}$$
	for every $i,j\geq0$, we conclude, by breaking this sequence into two short exact sequences, that $E_\infty^{p,q}$ is $I$-weakly cofinite for all $p,q\geq0$. The result now follows by convergence.\end{proof}
	
	The next statement, which is concerned with ordinary local cohomology, is new and will be a key tool in the proof of Theorem \ref{dim+gpd=4}.
	
	\begin{theorem}\label{criteria dim < 2}
	 Suppose $R$ is local and $\dim N \leq 2$. The following conditions are equivalent:
	\begin{itemize}
	    \item[(\rm i)] $H^i_{I} (N)$ is $I$-weakly cofinite for all $i \geq 0$;
	    \item[(\rm ii)] $\Ext^i_R (R/I, N)$ is weakly Laskerian for all $i \geq 0$;
	    \item[(\rm iii)] $\Ext^i_R (R/I, N)$ is weakly Laskerian for $i=0,1$.
	\end{itemize}
	\end{theorem}
	\begin{proof}
	Taking $\mathcal{S}$ as the category of weakly Laskerian $R$-modules in \cite[Proposition 3.9]{Mel}, we see that (i) implies (ii).
	It is clear that (ii) implies (iii). 
	Now, assume (iii) and let us prove that (i) holds. 
	
	Write ${\bf m}$ for the maximal ideal of $R$. First, by the flat base change property (see \cite[Theorem 4.3.2]{Brodman}) together with \cite[Lemma 2.1]{Marley}, we may assume that $R$ is ${\bf m}$-adically   complete. Given two non-negative integers $i,j$, let us consider the module $$T=\Ext^i_R (R/I, H^j_I (N)).$$ For any given $R$-submodule $L$ of $T$, we need to deduce that $\Ass_R T/L$ is finite. Suppose, by way of contradiction, that this set is infinite for some submodule $L$. So, we can choose  a countably infinite subset $$\displaystyle\ \{P_t\}_{t=1}^{\infty}\subset \Ass_R T/L, \quad P_t \neq {\bf m} \quad \mbox{for \,each} \quad t.$$ Consider the multiplicative closed subset $S = R\setminus \bigcup_{t=1}^{\infty} P_t$. Then, for every $t$, we have $S^{-1}P_t \in \Ass_{S^{-1}R} S^{-1}T/S^{-1}L$. Notice, in particular, that  $\Ass_{S^{-1}R} S^{-1}T/S^{-1}L$ is infinite. On the other hand, \cite[Lemma 3.2]{MV} (where $R$ is required to be complete) ensures that ${\bf m}\nsubseteq \bigcup_{t=1}^{\infty} P_t$, i.e,  ${\bf m}$ meets $S$, which (as ${\rm dim}\,N\leq 2$) forces $$\dim  S^{-1}N \leq 1.$$ Now, since the $S^{-1}R$-module $\Ext^i_{S^{-1}R} (S^{-1}R/S^{-1}I, S^{-1}N)$ is  weakly Laskerian for $i=0,1$, we can apply \cite[Proposition 3.4]{RM} in order to obtain that $H^j_{S^{-1}I} (S^{-1}N)$ is $S^{-1}I$-weakly cofinite for all $j \geq 0$. It follows that   $S^{-1}T$ is a weakly Laskerian $S^{-1}R$-module and hence so is the quotient $S^{-1}T/S^{-1}L$. In particular, $\Ass_{S^{-1}R} S^{-1}T/S^{-1}L$ is finite, which is a contradiction.
	\end{proof}
	
	\begin{question}\rm A problem which is now natural (and whose technical difficulties we have been unable to circumvent) is to decide, for a local ring $R$ and  $R$-modules $M, N$ with $M$ finite and $\dim  N \leq 2$, the equivalence of the following assertions:
	\begin{itemize}
	    \item[(\rm i)] $H^i_{I} (M, N)$ is $I$-weakly cofinite for all $i \geq 0$;
	    \item[(\rm ii)] $\Ext^i_R (M/IM, N)$ is weakly Laskerian for all $i \geq 0$;
	    \item[(\rm iii)] $\Ext^i_R (M/IM, N)$ is weakly Laskerian for $i=0,1$.
	\end{itemize}
	\end{question}

	
	
	
	
	

Now let us state the main result of this section.

\begin{theorem}\label{dim+gpd=4} Suppose $R$ is local and $M, N$ is an $I$-cohomologically injective pair of finite $R$-modules, with $\rho(M, N)=4$ $($see Notation \ref{delta}$)$.
If $\Hom_R (R/ I, H^2_{I} (\Hom_R (M,N)))$ is weakly Laskerian, then the $R$-module $H^i_{I} (M,N)$ is $I$-weakly cofinite for all $i\geq 0$. In particular, $\Ass_R H^i_{I} (M,N)$ is finite for all $i\geq 0$.
\end{theorem}
\begin{proof} Set $T= \Hom_R (M,N)$. We can write $\dim  T \leq\dim  M\otimes_RN\leq \rho(M, N)=4$. By Theorem \ref{gpdfinite}(iii), we can assume $\dim  T=4$. Therefore, we get
$$\dim   T=\dim  M\otimes_RN=4 \quad
\mbox{and} \quad \Gpd_NM=0.$$ Note the latter equality means ${\rm Ext}_R^j(M, N)=0$ for all $i\geq 1$.
Now, in particular, the spectral sequence of Lemma \ref{ss}(i) ensures that $$H^i_I(M,N)\cong H^i_I(T) \quad \mbox{for \,all} \quad i\geq 0.$$ Recall that, by the proof of Theorem \ref{gpdfinite}(ii), $H^i_I(M,N)=0$ for all $i>4$ and in addition the modules $H^4_I(M,N)\cong H^4_I(T)$ and $H^3_I(M,N)\cong H^3_I(T)$ are both $I$-weakly cofinite (also recall $H^0_I(M,N)$ always has this property). It thus suffices to prove the $I$-weak cofiniteness of $H^2_{I} (T)$. This will be accomplished by means of Corollary \ref{corollary melkersson weakly cofinite}.
        
    First, since $H^2_I(T)\cong H^2_I(T/H^0_{I}(T))$, we may suppose that $\depth_IT>0$ and choose a $T$-regular element $x \in I$. This gives an exact sequence
    $$\xymatrix@=1em{
     H^1_{I} (T) \ar[r]& H^1_{I} (T) \ar[r]^{\psi}& H^1_{I} (T/xT) \ar[r]& H^2_{I} (T) \ar[r]^{\varphi}& H^2_{I} (T) \ar[r]& H^2_{I} (T/xT) \ar[r]& H^3_{I} (T),}
    $$
    where $\varphi$ is  multiplication by $x$ on $H^2_{I} (T)$.
    As $\dim  T/xT =3$, we can apply \cite[Corollary 5.3]{CGH} to the pair $R, T/xT$ in order to obtain that the ordinary local cohomology modules $H^i_{I} (T/xT)$ are $I$-weakly cofinite for all $i\geq 0$. Since moreover $H^3_{I} (T)$ is $I$-weakly cofinite, the module $\coker \varphi = H^2_I(T)/xH^2_I(T)$ must have the same property. Also note the submodule $\Hom_R (R/I, \ker \varphi)\subset \Hom_R (R/I, H_I^2(T))$ is weakly Laskerian.


Now, let us consider the short exact sequence $$\xymatrix@=1em{
0\ar[r] & \im \psi\ar[r] & H^1_{I} (T/xT)\ar[r] & \ker \varphi\ar[r] & 0.
}$$ In particular, $\Hom_R (R/I, \im \psi)\subset \Hom_R (R/I, H_I^1(T/xT))$ is weakly Laskerian, and by means of the induced exact sequence  $$\xymatrix@=1em{
\Hom_R (R/I, H^1_{I} (T/xT))\ar[r] & \Hom_R (R/I, \ker \varphi)\ar[r] & \Ext^1_R (R/I, \im \psi)}$$ we obtain that $\Ext^1_R (R/I, \im \psi)$ is weakly Laskerian as well.
Therefore, we have shown that $\Ext^j_R (R/I, \im \psi)$ is weakly Laskerian for $j=0,1$. On the other hand, using \cite[Theorem 1.1(ii)]{SM} we get $\dim  H^1_{I} (T/xT) \leq \dim  T/xT -1 = 2$. It follows by Theorem \ref{criteria dim < 2} that $\im \psi$ is $I$-weakly cofinite, which forces the same property on the module $\ker \varphi = 0:_{H^2_I(T)}x$. Applying Corollary \ref{corollary melkersson weakly cofinite}, we conclude that $H^2_I(T)$ is $I$-weakly cofinite, as needed.
\end{proof}



\begin{remark}\label{dimhom4remark}\rm (i) The isomorphisms $H^i_I(M,N)\cong H^i_I(\Hom_R(M,N))$ for all $i\geq0$, which appear in the proof of Theorem \ref{dim+gpd=4}, hold in a wide range of situations. Suppose, for example, that $R$ is local, $\depth_RM=\depth R$, and $N$ has finite injective dimension. Then the well-known Ischebeck's formula (see, e.g., \cite[Exercise 3.1.24]{BH}) yields $\Gpd_NM=0$. Now the desired isomorphisms follow by convergence in the spectral sequence of Lemma \ref{ss}(i).

\medskip

\noindent (ii) By the very first part of the proof of Theorem \ref{dim+gpd=4} above, along with Theorem \ref{gpdfinite}(iii) and \cite[Lemma 2.2(iv)]{DAH}, we deduce the following fact. Suppose $R$ is local and $N$ is a finite $R$-module. If there exists a non-free finite $R$-module $M$ with $\pd_RM<\infty$ such that the pair $M, N$ is $I$-cohomologically injective with $\rho(M, N)=4$, then $H^i_{I} (M,N)$ is $I$-weakly cofinite (in particular, $\Ass_R H^i_{I} (M,N)$ is finite) for all $i\geq 0$.

\end{remark}

For ordinary local cohomology, we derive the following result which complements Corollary \ref{dim4-0}.


\begin{corollary}\label{corollary four dimensional}
Suppose $R$ is local and $N$ is finite with $\dim N=4$. The following conditions are equivalent:
\begin{itemize}
    \item[\rm (i)] $H^i_I (N)$ is $I$-weakly cofinite for all $i\geq 0$;
     \item[\rm (ii)] $H^i_I (H_I^2(N))$ is $I$-weakly cofinite for all $i\geq 0$;
    
    \item[\rm (iii)] $\Hom_R (R/I, H^2_I (N))$ is weakly Laskerian.
\end{itemize}
\end{corollary}
\begin{proof} By virtue of \cite[Theorem 1.1(ii)]{SM}, we have $\dim H^2_I (N) \leq 2$. Then, 
Theorem \ref{criteria dim < 2} ensures that (ii) holds if and only if  $\Ext^j_R (R/I, H^2_I (N))$ is weakly Laskerian for all $j\geq 0$, i.e.,
$H^2_I (N)$ is $I$-weakly cofinite, which in turn implies condition (iii). Now, noticing that $$\rho(R, N)=\dim N+\Gpd_NR=\dim N=4$$ and applying Theorem \ref{dim+gpd=4} with $M=R$, we conclude that (iii) implies (i). \end{proof}


We close the section with an application of Theorem \ref{criteria dim < 2} concerning the category $\mathcal{C}^{2} (R, I)_{\rm wcof}$ of all $I$-weakly cofinite $R$-modules of dimension at most 2. Our result extends \cite[Theorem 3.5]{BAH}, which considered the category $\mathcal{C}^{1}(R, I)_{\rm wcof}$.
	
	\begin{corollary}\label{theorem abelian category}
    If $R$ is local then $\mathcal{C}^{2} (R, I)_{\rm wcof}$ is an Abelian category.
    \end{corollary}
    
    \begin{proof}
    Let $f:M \rightarrow N$ be an $R$-linear map, where $M,N$ are objects in $\mathcal{C}^{2} (R, I)_{\rm wcof}$. The class of all $R$-modules with dimension at most 2 is a Serre subcategory of the category of all $R$-modules. So, the modules $\ker f$, $\im f$ and $\coker f$ have dimension at most 2. Note the submodule
    $\Hom_R (R/I, \ker f) \subset \Hom_R (R/I, M)$ is weakly Laskerian. Similarly, $\Hom_R (R/I, \im f)$ has this property as well. On the other hand, the short exact sequence
    $\xymatrix@=1em{
0\ar[r] &  \ker f\ar[r] & M\ar[r] & \im f\ar[r] & 0
}$ yields the exact sequence
$$\xymatrix@=1em{
\Hom_R (R/I, \im f)\ar[r] & \Ext^1_R (R/I, \ker f)\ar[r] &\Ext^1_R (R/I, M),}$$ from which we get the weak Laskerianess of the module $\Ext^1_R (R/I, \ker f)$.
    Hence, by Theorem \ref{criteria dim < 2}, $\ker f$ is $I$-weakly cofinite, and so by the exact sequence
    $$\xymatrix@=1em{
\Ext^1_R (R/I, M)\ar[r] & \Ext^1_R (R/I, \im f)\ar[r] &\Ext^2_R (R/I,\ker f)}$$
we deduce that $\Ext^1_R (R/I, \im f)$ is weakly Laskerian. Consequently, applying again Theorem \ref{criteria dim < 2}, $\im f$ must be $I$-weakly cofinite. Finally, the fact that $\coker f$ is $I$-weakly cofinite follows from the exact sequence 
$\xymatrix@=1em{
0\ar[r] &  \im f\ar[r] & N\ar[r] & \coker f\ar[r] & 0.}$\end{proof}

\section{A pathological example}\label{patho}

In this last section we provide two examples which, in our view, are of interest in regard to the theory investigated in this note. They illustrate that, if $M$ is finite and $N$ is {\it not} weakly Laskerian (in particular, non-finite), then the set $\Ass_R H^1_{I}(M,N)$ can be infinite or finite. The latter instance, in particular, reveals a rather pathological behavior which, to the best of our knowledge, has not been noticed in the literature.
In both examples, $R=k[x_1,\ldots, x_n]$ denotes a polynomial ring in indeterminates $x_1, \ldots, x_n$ over an infinite field $k$.

\begin{example}\rm\label{example infinite}
 Assume $n\geq 3$ and let $I=(x_1)\subset R$. Consider the prime ideals $$P_t = (x_2,\ldots,x_{n-1},x_{n}-t) \quad \mbox{for \,each} \quad t\in k.$$ Set $M= R/(x_2)$ and  $N = \displaystyle\bigoplus_{t \in k} R/{P_t}$. First of all, notice that $N$ is not weakly Laskerian since $P_t \in \Ass_R N$ for all $t \in k$. Moreover, since $H^1_I (M,N) \cong \displaystyle\bigoplus_{t \in k} H^1_I (M,R/{P_t})$, we obtain $\Ass_R H^1_I (M,R/{P_t}) \subset \Ass_R H^1_I (M,N)$ for each $t \in k$. On the other hand, because $$\depth_{\ann_RM/IM}R/{P_t}=\depth_{(x_1, x_2)}R/{P_t}=\depth_{I}R/{P_t}=1,$$ we can apply \cite[Theorem 2.4]{CH} in order to write $$\Ass_R \Hom_R (R/I, H^1_I (M,R/P_t)) = \Ass_R \Ext^1_R (M/IM, R/P_t)$$ $$= \Ass_R \Ext^1_R (R/(x_1,x_2), R/P_t)=\Ass_RR/((x_1, x_2)+P_t).$$ We then derive that
$$\Ass_R H^1_I (M,R/P_t) = \Ass_R \Hom_R (R/I, H^1_I (M,R/P_t)) = \{(x_1,\ldots,x_{n-1},x_{n}-t)\},$$
which shows that $\Ass_R H^1_I (M,N)$ is infinite.
 \end{example}


 \begin{example}\rm\label{example finite} Assume $n\geq 2$. Fix $\lambda_{0} \in k^{\ast}=k\setminus \{0\}$ and let $I=(x_1-\lambda_{0})\subset R$. Set $$N = \left( \displaystyle\bigoplus_{\lambda \in k^{\ast}} {k[x_1]}_{x_1-\lambda}\right) \displaystyle \oplus R_{x_1}.$$ This is an $R$-module with $x_i$ acting trivially on ${k[x_1]}_{x_1-\lambda}$ for each $i=2, \ldots,n$. We shall see, in the steps below, that $N$ is not weakly Laskerian but, quite surprisingly, the set $\Ass_R H^1_I (M,N)$ is finite for any finite $R$-module $M$.
 
 \smallskip

 \noindent  {\bf Step 1}. $N$ is not weakly Laskerian.
 
\smallskip
 
 For each $\lambda \in k^{\ast}$, set $T^{\lambda} =  {k[x_1]}_{x_1-\lambda}$ and  ${\bf m}_{\lambda}=(x_1-\lambda, x_2, \ldots,x_n)\subset R$. Notice that 
 $${\bf m}_{\lambda} = \ann_R \overline{(x_1-\lambda)^{-1}} \in   \Ass_R {T^{\lambda}}/ k[x_1].$$
 On the other hand, by taking $L = \left( \displaystyle\bigoplus_{\lambda \in k^{\ast}} k[x_1]\right) \displaystyle\oplus R\subset N$, we obtain $$
 N/L \cong \left( \displaystyle\bigoplus_{\lambda \in k^*} T^{\lambda}/ k[x_1]\right) \displaystyle\oplus {R_{x_1}}/ R$$
 and hence ${\bf m}_{\lambda} \in \Ass_RN/L$, for all $\lambda \in k^{\ast}$. Therefore, $\Ass_RN/L$ is infinite. In particular, $N$ is not weakly Laskerian.


 \smallskip

 \noindent  {\bf Step 2}. $\Ass_R H^1_I (M,L)$ is finite.
 
\smallskip
 
The short exact sequence
$\xymatrix@=1em{
0\ar[r] &  L\ar[r] & N\ar[r] & N/L\ar[r] & 0}$ induces an exact sequence
 $$\xymatrix@=1em{H^0_I (M,N/L) \ar[r] & H^1_I (M,L)\ar[r] & H^1_I (M,N)\ar[r] &  H^1_I (M,N/L),}$$ which in turn yields the inclusion
 $$
 \Ass_R H^1_I (M,N) \subset \Ass_R H^1_I (M,L) \cup \Supp_R H^0_I (M,N/L) \cup \Supp_R H^1_I (M,N/L). 
 $$
Now notice that $H^1_I (M,L) \cong H^1_I \left(M,\displaystyle\bigoplus_{\lambda \in k^{\ast}} k[x_1]\right) \displaystyle\oplus H^1_I (M,R)$. By Proposition \ref{prop Ass H^1(M,N)}, the set $\Ass_R H^1_I (M,R)$ is finite. Furthermore, given $P \in \Ass_R H^1_I \left(M,\displaystyle\bigoplus_{\lambda \in k^{\ast}} k[x_1]\right) $, we have $I \subset P$ and $P \in \Supp_{R} k[x_1]$. Therefore $$x_2, \ldots, x_n \in 0:_{R}k[x_1]\subset P,$$ so that
 $(x_1-\lambda_0, x_2,\ldots,x_n)\subset P$ and hence $P={\bf m}_{{\lambda}_0}$, which gives $$\Ass_{R} H^1_I \left(M,\displaystyle\bigoplus_{\lambda \in k^{\ast}} k[x_1]\right) \subset \{{\bf m}_{{\lambda}_0}\}.$$ Thus we have shown that $\Ass_R H^1_I (M,L)$ is finite.

  \smallskip

 \noindent  {\bf Step 3}. $\Ass_R H^1_I (M, N)$ is finite.
 
\smallskip

It is clear that each element of the quotient $T^{\lambda}/k[x_1]$ is annihilated by $x_2, \ldots, x_n$ and by some power of $(x_1-\lambda)$. This shows that ${\bf m}_{\lambda} \subset Q$ for each $Q \in \Supp_R T^{\lambda}/k[x_1]$, i.e.,
 $$
 \Supp_R T^{\lambda}/k[x_1] = \{{\bf m}_{\lambda}\}.$$
 In a similar fashion, we see that $(x_1) \subset Q$ whenever $Q \in \Supp_R R_{x_1}/ R$.
 Now, given $P \in \Supp_R H^i_I (M,N/L)$, we must have $I \subset P$ and $P \in \Supp_RN/L$. Since
 $$
 \Supp_RN/L = \left(\displaystyle\bigcup_{\lambda \in k^{\ast}} \Supp_RT^{\lambda}/ k[x_1]\right) \displaystyle\cup \Supp_R R_{x_1} / R,$$
 we get that either $P={\bf m}_{\lambda}$ for some $\lambda \in k^{\ast}$, or
 $P \in \Supp_R R_{x_1} /R$.
 The latter case is impossible; indeed, supposing otherwise, we would have $(x_1)\subset P$ as seen above, whereas also $I=(x_1-{\lambda}_0) \subset P$, a contradiction as $\lambda_0\neq 0$. Thus, $P={\bf m}_{\lambda}$. Since again $I \subset P$, we necessarily have $P={\bf m}_{{\lambda}_0}$. Hence $$\Supp_R H^i_I (M,N/L)=\{{\bf m}_{{\lambda}_0}\},$$ which, using the previous step and its proof, guarantees that  $\Ass_R H^1_I (M,N)$ is finite, as claimed.

 For completeness, notice that by taking $M=R$ we obtain the same property for the set $\Ass_RH^1_I(N)$ (the module $H^1_I(N)$ is non-zero by the remark below).
 
 \end{example}

\begin{remark}\rm The goal of Example \ref{example finite} would be essentially meaningless if it could happen that $H^1_I (M,N)=0$ for  every finite $R$-module $M$, so now it is worth showing that this is {\it not} the case. Indeed, we claim that $H^1_I (M,N) \neq 0$ whenever $M$ is a finite $R$-module satisfying
$$\Hom_R (M, k[x_1]_{x_1- \lambda_0} / k[x_1]) \neq 0,$$ which is easily seen to hold if, e.g., $M$ is free or $M= k[x_1, \ldots,x_r]$ for $1\leq r \leq n$. In particular, taking $M=R$ we get that $H_I^1(N)\neq 0$.

To prove the claim, suppose by way of contradiction that $H^1_I (M,N) =0$ for any $M$ as above. Then, $H^1_I (M, k[x_1]_{x_1-\lambda})= 0$ for all $\lambda \in k^{\ast}$. Also note that, because clearly $H^0_I(k[x_1]_{x_1-\lambda})=0$, we have $H^0_I (M, k[x_1]_{x_1-\lambda})\cong \Hom_R (M, H^0_I (k[x_1]_{x_1-\lambda})) =0$. Now, from the short exact sequence $$\xymatrix@=1em{
0\ar[r] &  k[x_1]\ar[r] & k[x_1]_{x_1-\lambda}\ar[r] & k[x_1]_{x_1-\lambda}/ k[x_1]\ar[r] & 0}$$
we derive isomorphisms
$$H^1_I (M, k[x_1]) \cong H^0_I (M, k[x_1]_{x_1-\lambda}/ k[x_1]) \cong \Hom_R (M, H^0_I (k[x_1]_{x_1-\lambda}/ k[x_1])).$$ If we pick any $\lambda \neq \lambda_0$, we immediately deduce $H^0_I (k[x_1]_{x_1-\lambda}/ k[x_1]) =0$, which shows that $H^1_I (M, k[x_1])=0$. Now, taking $\lambda = \lambda_0$ (in which case the module $k[x_1]_{x_1- \lambda_0} / k[x_1]$ is $I$-torsion), we get
 $$\Hom_R (M, k[x_1]_{x_1- \lambda_0} / k[x_1])\cong H^1_I (M, k[x_1])= 0.$$ This contradiction proves the claim and closes the paper.
 
 
 \end{remark}



\bigskip

\noindent{\bf Acknowledgements.} The first and second-named authors were supported by CAPES Doctoral Scholarships. The third-named author was partially supported by the CNPq-Brazil grants 301029/2019-9 and 406377/2021-9.


\begin{thebibliography}{9}
\bibliographystyle{alpha}

\bibitem{JN}
J. Amjadi, R. Naghipour, \emph{Cohomological dimension of generalized local cohomology modules}, Algebra Colloq. {\bf 2} (2008), 303--308.


\bibitem{Bah0} K. Bahmanpour, {\it On the category of weakly Laskerian cofinite modules}, Math. Scand. {\bf 115} (2014), 62--68.



\bibitem{BAH1} K. Bahmanpour, R. Naghipour, \emph{Cofiniteness of local cohomology modules for ideals of small dimension},  J. Algebra {\bf 321} (2009), 1997--2011.


\bibitem{BAH}
K. Bahmanpour, R. Naghipour, M. Sedghi, \emph{Modules cofinite and weakly cofinite with respect to an ideal}, J. Algebra Appl. {\bf 16}  (2018), 1850056.

\bibitem{BBLSZ} B. Bhatt, M Blickle, G. Lyubeznik, A. K. Singh, W. Zhang, 
{\it Local cohomology modules of a smooth ${\mathbb Z}$-algebra
have finitely many associated primes}, Invent. Math. {\bf 197} (2014), 509--519.


\bibitem{B}
M. H. Bijan-Zadeh, {\it A common generalization of local cohomology theories}, Glasgow Math. J. \textbf{21} (1980), 173--181.


\bibitem{BL} M. P. Brodmann, F. A. Lashgari, {\it A finiteness result for associated primes of local
cohomology modules}, Proc. Amer. Math. Soc. {\bf 128} (2000), 2851--2853.


\bibitem{Brodman} M. P. Brodmann, R. Y. Sharp,  {\it Local Cohomology. An Algebraic Introduction with Geometric Applications}, 2nd ed., Cambridge Stud. Adv. Math.  \textbf{136}, Cambridge Univ. Press, Cambridge, 2013.

\bibitem{BH}
W. Bruns, J. Herzog, {\it Cohen–Macaulay rings}, revised edition, Cambridge Stud. Adv. Math. \textbf{39},
Cambridge Univ. Press, Cambridge, 1998.


\bibitem{CGH}
N. T. Cuong, S. Goto, N. V. Hoang, \emph{On the cofiniteness of generalized local cohomology
modules}, Kyoto J. Math. {\bf 55} (2015), 169--185.

\bibitem{CH} N. T. Cuong, N. V. Hoang, {\it Some finiteness properties of generalized local cohomology modules}, East-West J. Math. {\bf 7} (2005), 107--115.

\bibitem{DAH}
K. Divaani-Aazar, A. Hajikarimi, \emph{Generalized local cohomology modules and homological Gorenstein dimensions}, Comm. Algebra \textbf{39} (2011), 2051--2067.

\bibitem{DAH1}
K. Divaani-Aazar, A. Hajikarimi,
\emph{Cofiniteness of generalized local cohomology modules for one-dimensional ideals}, Canad. Math. Bull. {\bf 55} (2012), 81--87.

\bibitem{DAM}
K. Divaani-Aazar, A. Mafi, \emph{Associated primes of local cohomology modules of weakly Laskerian modules}, Comm. Algebra \textbf{34} (2006), 681--690.


\bibitem{FJMS}
T. H. Freitas, V. H. Jorge-Pérez, C. B. Miranda-Neto, P. Schenzel, \emph{Generalized local duality, canonical modules, and prescribed bound on projective dimension}, J. Pure Appl. Algebra {\bf 227} (2023), 107188.

\bibitem{G} A. Grothendieck , \emph{Cohomologie locale des faisceaux cohérents et théorèmes de Lefschetz locaux et globaux
$($SGA 2$)$}, Adv. Stud. Pure Math. {\bf 2}, North-Holland Publishing Co., Amsterdam, 1968.

\bibitem{H0} R. Hartshorne, {\it Cohomological dimension of algebraic varieties}, Ann. of Math. {\bf (2)} {\bf 88}
(1968), 403--450.

\bibitem{H} R. Hartshorne, {\it Affine duality and cofiniteness}, Invent. Math. {\bf 9} (1970), 145--164.

\bibitem{Herzog} J. Herzog, {\it Komplexe, Aufl\"osungen und dualit\"at in der localen Algebra}, Habilitationsschrift, Universit\"at Essen, 1970.

\bibitem{Hochster} M. Hochster, \emph{Finiteness properties and numerical behavior of local cohomology}, Comm. Algebra {\bf 47} (2019), 1--11.



\bibitem{Hunekeconjecture} C. Huneke, \emph{Problems on local cohomology, in: Free resolutions in commutative algebra and
algebraic geometry}, Sundance, Utah, 1990, Res. Notes Math. {\bf 2},
Jones and Bartlett, Boston, MA, 1992, pp. 93--108.


\bibitem{Huneke} C. Huneke, J. Koh, \emph{Cofiniteness and vanishing of local cohomology modules},  Math. Proc. Cambridge Philos. Soc. {\bf 110} (1991), 421--429.


\bibitem{Huneke-L} C. Huneke, G. Lyubeznik, \emph{On the vanishing of local cohomology modules},  Invent. Math. {\bf 102} (1990), 73--93.

\bibitem{Huneke-S} C. Huneke, R. Y.    Sharp, {\it Bass numbers of local cohomology modules},
Trans. Amer. Math. Soc. {\bf 339} (1993), 765--779.

\bibitem{Ka} M. Katzman, {\it An example of an infinite set of associated primes of a local cohomology
module}, J. Algebra {\bf 252} (2002), 161--166.


\bibitem{Lyu0} G. Lyubeznik, {\it Finiteness properties of local cohomology modules $($an application of D-modules to commutative algebra$)$}, Invent. Math. {\bf 113} (1993), 41--55.


\bibitem{Lyu} G. Lyubeznik, {\it A partial survey of local cohomology, in: Local Cohomology
and Its Applications}, Marcel Dekker, 2001.



\bibitem{AM}
A. Mafi, \emph{A generalization of the finiteness problem in local cohomology modules}, Proc. Indian Acad. Sci. Math. Sci. {\bf 119} (2009),  159--164.

\bibitem{MS}
A. Mafi, H. Saremi, \emph{Cofinite modules and generalized local cohomology}, Houston J. Math. {\bf 35} (2009), 1013--1018.

\bibitem{SM} A. Mafi, H. Saremi, {\it On the finiteness dimension of local cohomology modules}, Algebra Colloq. {\bf 21} (2014),  517--520.

\bibitem{Marley} T. Marley, {\it The associated primes of local cohomology modules over rings of small dimension}, Manuscripta Math. {\bf 104} (2001), 519--525.

\bibitem{MV}
T. Marley, J. Vassilev, \emph{Cofiniteness and associated primes of local cohomology modules}, J. Algebra \textbf{256} (2002), 180--193.

\bibitem{Mel}
L. Melkersson, \emph{Modules cofinite with respect to an ideal}, J. Algebra. \textbf{285} (2005), 649--668.

\bibitem{RM} 
H. Roshan-Shekalgourabi, M.Hatamkhani, \emph{Hartshorne's question and weakly cofiniteness}, Math. Rep. (Bucur.) {\bf 22(72)} (2020),  {\it 3-4}, 329--340.



\bibitem{Sharp} R. Y. Sharp, {\it The Cousin complex for a module over a commutative Noetherian ring}, Math. Z. {\bf 112}
(1969), 340--356.

\bibitem{Sharp2} R. Y. Sharp, {\it Gorenstein Modules}, Math. Z. {\bf 115} (1970), 117--139.

\bibitem{Si} A. K. Singh, {\it $p$-torsion elements in local cohomology modules}, Math. Res. Lett. {\bf 7} (2000), 165--176.

\bibitem{Si2} A. K. Singh, {\it
Associated primes of local cohomology modules}, in: Proceedings of the 48th Algebra Symposium, Nagoya University (2003), 133--145.

\bibitem{SS} A. K. Singh, I. Swanson, {\it Associated primes of local cohomology modules and of
Frobenius powers}, Int. Math. Res. Not. {\bf 33} (2004), 1703--1733.



\bibitem{Y}
K. Yoshida, {\it Tensor products of perfect modules and maximal surjective Buchsbaum modules}, J. Pure Appl. Algebra {\bf 123} (1998), 313--323.

\end{thebibliography}
\end{document}